\documentclass[10pt]{amsart}
\usepackage[utf8]{inputenc}
\usepackage{amsmath,amssymb,amsthm}
\usepackage{mathrsfs}
\usepackage{enumitem}
\usepackage[normalem]{ulem}
\usepackage{xcolor}
\definecolor{RED}{named}{red}
\newtheorem{theorem}{Theorem}[section]
\newtheorem{lemma}[theorem]{Lemma}
\theoremstyle{definition}
\newtheorem{definition}[theorem]{Definition}
\newtheorem{remark}[theorem]{Remark}

\usepackage{amsmath}
\usepackage{amsfonts}
\usepackage{amssymb}

\usepackage{amsxtra}

\usepackage{amsthm}
\usepackage{amsmath,amssymb,amsfonts,amsthm,color,latexsym}
\usepackage{xcolor}
\usepackage{float}

\newcommand{\C}{\mathbb{C}}
\newcommand{\F}{\mathcal{F}}

\newcommand{\B}{\mathcal{B}}
\newcommand{\D}{\mathcal{D}}

\newtheorem{example}[theorem]{Example}

\newtheorem{proposition}[theorem]{Proposition}
\newtheorem{corollary}[theorem]{Corollary}

\usepackage[pagebackref]{hyperref}

\usepackage{todonotes}

\title[On Milnor and Tjurina numbers of Pairs of Holomorphic Function Germs]{On Milnor and Tjurina numbers of Pairs of Holomorphic Function Germs}

\date{\today}

\author[A. Fern\'andez-P\'erez]{Arturo Fern\'andez-P\'erez}
\address[Arturo Fern\'{a}ndez P\'erez] {Department of Mathematics. Federal University of Minas Gerais. Av. Ant\^onio Carlos, 6627 
CEP 31270-901\\
Pampulha - Belo Horizonte - Brazil. ORCID: 0000-0002-5827-8828}
\email{fernandez@ufmg.br}

\author[E. R.  Garc\'{i}a Barroso]{Evelia R. Garc\'{i}a Barroso}
\address[Evelia R. Garc\'{i}a Barroso]{Dpto. Matem\'{a}ticas, Estad\'{\i}stica e Investigaci\'on Operativa\\
Instituto Universitario de Matem\'aticas y Aplicaciones (IMAULL)\\
Universidad de La Laguna. Apartado de Correos 456. 38200 La Laguna, Tenerife, Spain. ORCID: 0000-0001-7575-2619}
\email{ergarcia@ull.es}

\author[N. Saravia-Molina]{Nancy Saravia-Molina}
\address[Nancy Saravia-Molina]{Dpto. Ciencias - Secci\'{o}n Matem\'{a}ticas, Pontificia Universidad Cat\'{o}lica del Per\'{u}, Av. Universitaria 1801,
San Miguel, Lima 32, Peru. ORCID: 0000-0002-2819-8835}
\email{nsaraviam@pucp.edu.pe}

\subjclass[2020]{Primary 32S65 - 14H20}

\keywords{Singular holomorphic foliations, meromorphic functions, Tjurina number of pairs, Milnor number of pairs, pencil of plane curves} 

\thanks{
The authors gratefully acknowledge the support of Universidad de La Laguna (Tenerife, Spain), where part of this work was carried out. This work was partially supported by  the Spanish grant PID2023-149508NB-I00, funded by 
MICIU/AEI
/10.13039/501100011033 and by
FEDER, UE. It was also supported by  the Vicerrectorado de Investigación
(VRI) at the PUCP through grant DFI-2025-PI1275. The first author also acknowledges the support of CNPq-Brazil through Projeto Universal 408687/2023-1 "Geometria das Equa\c{c}\~oes Diferenciais Alg\'ebricas" and the PQ fellowship CNPq-Brazil 306011/2023-9}

\begin{document}
\maketitle
\begin{abstract}
Motivated by the bifurcation formula
for the Milnor number of pairs of
holomorphic function germs, we
introduce, via foliation theory, the
Tjurina number of such pairs. We
define it as the
generic value of the Tjurina number of
the foliation
associated with the pair, with respect
to a specified family of balanced
divisors of separatrices. We prove
that this value is well defined
and that it is an
analytic invariant,
we derive an explicit
formula for it, and we
establish a bifurcation formula. As an
application, we characterize semitame
pairs in terms of this
invariant. We also
prove a version of
Teissier's lemma for pairs and an
upper bound for the Milnor number of a
pair, and we apply
these results to generic pencils of
algebraic curves.

\end{abstract}

\section{Introduction}

The study of numerical invariants
associated with a holomorphic function
$f\in\mathbb{C}\{x,y\}$ having an
isolated singularity has been a central
topic in Singularity Theory. Among the
most important such invariants are the
Milnor number $\mu(f)$ \cite{Milnor}
and the Tjurina number $\tau(f)$
\cite{Tjurina}. By contrast, numerical
invariants associated with a
meromorphic function $f/g$ have
received much less attention.
Recently, A.~Szawlowski \cite{Sz1}
introduced the Milnor number $\mu(f,g)$ of a pair
$(f,g)$. This invariant can be
interpreted as the Milnor number of the
meromorphic function $f/g$, since the
classical Milnor number is recovered
when $g$ is a unit in
$\mathbb{C}\{x,y\}$. Szawlowski's
approach is based on the singular
foliation $\F_{\omega_{f/g}}$, defined
by the $1$-form
$\omega_{f/g}=g\,df-f\,dg$, naturally
associated with the pencil generated by
$f$ and $g$. Motivated by this idea,
the aim of the present paper is to
introduce the Tjurina number of a pair
$(f,g)$, which may also be regarded as
the Tjurina number of the meromorphic
function $f/g$. Our main tool is the
Tjurina number of a foliation $\F$ with
respect to an $\F$-invariant reduced
curve, introduced in
\cite[Section~6]{FS-GB-SM1}, which we
apply to $\F_{\omega_{f/g}}$ in several
contexts. More precisely, in Section \ref{sec: Tjurina-pair-fiber}
we define the Tjurina number of the
pair $(f,g)$ with respect to a member
$C_{[\alpha:\beta]}$ of the pencil of
curves
\[
\mathcal{P}_{f,g}:
C_{[\alpha:\beta]}=
\{\alpha f+\beta g=0\},\qquad
[\alpha:\beta]\in\mathbb{P}^1,
\]
as the Tjurina number of
$\F_{\omega_{f/g}}$ with respect to the
invariant curve $C_{[\alpha:\beta]}$,
and we relate it to the Tjurina number
of $C_{[\alpha:\beta]}$ and to
$\mu(f,g)$. We also introduce the
generic Tjurina number of a pair
$(f,g)$, relate it to the generic
Tjurina number of the associated
pencil, and show by means of an example
(see Example~\ref{ex1}) that it is not
determined by the individual
topological types of the two generators.

\medskip
\noindent
In Section~\ref{sec: pencil_Tjurina} we
consider balanced divisors of
separatrices of $\F_{\omega_{f/g}}$
whose zero divisor is supported on the
bifurcation fibers together with the
fibers defined by $f$ and $g$, and
whose pole divisor is supported on
generic fibers of the pencil
$\mathcal{P}_{f,g}$. The main result of
this section (see Theorem~\ref{teo2})
provides an explicit formula for the
difference between the Tjurina numbers
of $\F_{\omega_{f/g}}$ with respect to
its zero and pole divisors.

\medskip
\noindent
In Section~\ref{sec:biju} we introduce
the Tjurina number
$\tau(\F,\mathcal{D})$ of a foliation
$\F$ with respect to a primitive
balanced divisor of separatrices
$\mathcal{D}$. Whereas the numbers of
Section~\ref{sec: Tjurina-pair-fiber}
depend on a single fiber, here we
restrict, for $\F=\F_{\omega_{f/g}}$,
to the family of divisors
$\mathcal{D}_{\mathbf{t}}$ adapted to
$fg=0$ whose zero divisor is the one of
Section~\ref{sec: pencil_Tjurina} and whose
pole divisor consists of $r$ generic
fibers of $\mathcal{P}_{f,g}$, where
$r$ is the number of bifurcation fibers
other than $f$ and $g$. We prove that
$\tau(\F,\mathcal{D}_{\mathbf{t}})$ is
constant, and maximal, for a generic
choice of these fibers, and we define
the Tjurina number $\tau(f,g)$ of the
pair $(f,g)$ as this generic value.
Restricting to this family is
essential: for arbitrary balanced
divisors adapted to $fg=0$ the Tjurina
number depends on the choice of the
divisor (see
Remark~\ref{rem:essential}).
Theorem~\ref{coro:bif} gives an
explicit formula for $\tau(f,g)$ in
terms of the Tjurina numbers of the
zero and pole divisors of
$\mathcal{D}_{\mathbf{t}}$, the number
$r$, and the intersection multiplicity
$I_0(f,g)$. Inspired by the Milnor
bifurcation formula of \cite{FFMR}, we
also establish a bifurcation formula
for $\tau(f,g)$ and characterize
semitame pairs in terms of this
invariant (see
Corollary~\ref{caracsemitame}).

\medskip
\noindent
In Sections~\ref{sec: Teissier}
and~\ref{sec: bound} we turn our
attention to the Milnor number
$\mu(f,g)$ of a pair of holomorphic
functions and investigate some of its
properties. More precisely,
Section~\ref{sec: Teissier} establishes
a version of Teissier's Lemma for pairs
of holomorphic functions, while in
Section~\ref{sec: bound} we derive an
upper bound for $\mu(f,g)$ in terms of
the Tjurina numbers of the bifurcation
fibers. Finally, in
Section~\ref{sec: Dimca} we give an
application to pencils of algebraic
curves, relating to a formula of Dimca
\cite{Dimca2017} for generic pencils.

 
\section{the Tjurina number of a pair with respect to a fiber of the associated pencil} \label{sec: Tjurina-pair-fiber}
\noindent Let $\mathcal{O}_2:=\mathbb C\{x,y\}$ denote the ring of convergent complex power series in two variables.  
\noindent Let $f,g\in \mathcal{O}_2$. The intersection multiplicity of the curves $f(x,y)=0$ and $g(x,y)=0$ at the origin is denoted by $I_0(f,g)$ and is defined by
$I_0(f,g)=\dim_{\mathbb C}\mathcal{O}_2/(f,g)$,
where $(f,g)$ denotes the ideal of $\mathcal{O}_2$ generated by $f(x,y)$ and $g(x,y)$. The Milnor number of a reduced plane curve $C:f(x,y)=0$ is by definition $\mu(C):=I_0(f_x,f_y)$ and  the Tjurina number of $C$ is $\tau(C):=\dim_{\mathbb C}\mathcal{O}_2/(f,f_x,f_y)$. 

Given relatively prime elements $f,g\in \mathcal{O}_2$ vanishing at the origin, we
consider the holomorphic foliation $\mathcal{F}_{\omega _{f/g}}$ defined by the $1$-form
\[
\omega _{f/g}:=gdf-fdg=(gf_x - fg_x)dx+(gf_y - fg_y)dy.
\]

We assume that the foliation $\mathcal{F}_{\omega_{f/g}}$ has an isolated singularity at $0\in\C^2$. Note that this condition is equivalent to finite determinacy of the meromorphic function $f/g$ in the sense of Cerveau-Mattei \cite[p. 147]{Cerveau}.
By \cite[Lemma 4.1]{FFMR}, all the fibers of the pencil are reduced.
\begin{remark}
   Since every fiber  of the pencil $\mathcal{P}_{f,g}:C_{[\alpha:\beta]} = \{\, \alpha f + \beta g = 0 \,\},  [\alpha:\beta]\in \mathbb{P}^1$, is a $\F_{\omega_{f/g}}-$invariant curve, the foliation is \textit{dicritical}. Moreover, it is a \textit{generalized curve} foliation, since  its reduction of singularities has no  saddle-nodes (for more details see \cite[Chapter II]{Camacho}). 
\end{remark}

\par Let $\mathcal{F}:\omega=0$ be a germ of holomorphic foliation at the origin defined by the $1$-form  $\omega=P(x,y)dx+Q(x,y)dy$ and let $C:h=0$ be a reduced plane curve. The {\it Tjurina number of $\mathcal{F}$ with respect to $C$}  was studied in \cite{FS-GB-SM1} and is defined by
\begin{equation}\label{def:Tj-fol-curve}
\tau(\mathcal{F},C):=\dim_{\mathbb C} \frac{\mathcal{O}_2}{(P,Q,h)}.
\end{equation}

\begin{definition} Consider the pencil $\mathcal{P}_{f,g}=\{\alpha f+\beta g\;:\;[\alpha:\beta]\in \mathbb{P}^1\}$ generated by $f,g \in \mathcal{O}_2$. The {\it Tjurina number of the pair} $(f,g)$ with respect to the fiber $C_{[\alpha:\beta]}: \alpha f+\beta g=0$ of $\mathcal{P}_{f,g}$ is defined as the Tjurina number of $\mathcal{F}_{\omega _{f/g}}$ with respect to the reduced curve $C_{[\alpha:\beta]}$, namely
\begin{equation}\label{Tpair}
\tau_{[\alpha:\beta]}(f,g):=\tau\left(\mathcal{F}_{\omega _{f/g}},C_{[\alpha:\beta]}\right)=\dim_{\mathbb C}  \frac{\mathcal{O}_2}{(gf_x-fg_x,gf_y-fg_y,\alpha f+\beta g)}.
\end{equation}
\end{definition}
\medskip

\noindent Denote by $GSV_p(\mathcal{F}, C)$ the GSV-index of $\mathcal{F}$ along the $\mathcal F$-invariant reduced curve $C$ at $p$. This is an analytic invariant of $\F$ and $C$ (see \cite{GSV-paper}).

\par Next we show some results on the Tjurina number of the pair $(f,g)$ with respect to a member of pencil $\mathcal{P}_{f,g}$. 
\begin{proposition}\label{Tju}
Let $f,g\in \mathcal{O}_2$ be relatively prime and let $C_{[\alpha:\beta]}: \alpha f+\beta g=0$,  $[\alpha:\beta] \in\mathbb{P}^1$   be a reduced curve. Then 
\[
\tau_{[\alpha:\beta]}(f,g)=I_0(f,g)+\tau\left(C_{[\alpha:\beta]}\right).
\]
\end{proposition}
\begin{proof} 
Applying \cite[Proposition 6.2]{FS-GB-SM1} to the $\mathcal{F}_{\omega_{f/g}}$-invariant curve $C_{[\alpha:\beta]}:\alpha f+\beta g =0$ we obtain
\begin{equation}\label{ec1}
GSV\left(\mathcal{F}_{\omega_{f/g}}, C_{[\alpha:\beta]}\right)=\tau\left(\mathcal{F}_{\omega_{f/g}} , C_{[\alpha:\beta]}\right) - \tau\left(C_{[\alpha:\beta]}\right). 
\end{equation} 
On the other hand \cite[Proposition 4.8]{FFMR} yields
$GSV(\mathcal{F}_{\omega_{f/g}}, C_{[\alpha:\beta]})=I_0(f,g)$. The result now follows from the definition of $\tau_{[\alpha:\beta]}(f,g)$ and  equality \eqref{ec1}.
\end{proof}

\begin{remark}
It is well known that $\mu(fg)=\mu(f)+\mu(g)+2I_0(f,g)-1$. Combining this identity with Proposition \ref{Tju}, we obtain
\[
2\tau_{[\alpha:\beta]}(f,g)=\mu(fg)-\mu(f)-\mu(g)+2\tau(C_{[\alpha:\beta]})+1.
\]
\end{remark}

\bigskip

The {\it Milnor number of the pair} $(f,g)$ was introduced by Szawlowski \cite{Sz1} and is defined by

\begin{equation}\label{Mpair}
\mu(f,g):=\mu\left(\mathcal{F}_{\omega_ {f/g}}\right)=\dim_{\mathbb C} \frac{\mathcal{O}_2}{(gf_x-fg_x,gf_y-fg_y)}.
\end{equation}

Note that   $\mu(f,g) \geq \tau_{[\alpha:\beta]}(f,g)$.

In \cite{Suwa}, Suwa  studied  the unfoldings of germs of meromorphic functions and introduced the invariant
\begin{equation}\label{def:nu}
\nu(f,g)=\dim_{\mathbb{C}}\frac{(f,g)}{(f_xg-g_xf,f_yg-g_yf)}.
\end{equation}

\medskip

\noindent Equations \eqref{Mpair} and \eqref{def:nu} imply that
\begin{equation}\label{im}
I_0(f,g)=\mu(f,g)-\nu(f,g).
\end{equation}

\begin{corollary}\label{ecc}
Let $f,g\in \mathcal{O}_2$ be relatively prime and let $C_{[\alpha:\beta]}: \alpha f+\beta g=0$,  $[\alpha:\beta] \in\mathbb{P}^1$   be a reduced curve. We have
\begin{equation}\nonumber   
\tau_{[\alpha:\beta]}(f,g)=\tau(C_{[\alpha:\beta]})+\mu(f,g)-\nu(f,g).
\end{equation}
\end{corollary}
\begin{proof}
It follows from \eqref{im} and Proposition \ref{Tju}.
\end{proof}

Denote by $\mu\left(\mathcal{F}_{\omega_{f/g}}, C_{[\alpha:\beta]}\right)$ the multiplicity of $\mathcal{F}_{\omega_{f/g}}$ along  the fiber $C_{[\alpha:\beta]}$ (see \cite[p. 152]{Camacho}). 
 \begin{corollary}\label{mu}
For every $[\alpha:\beta] \in \mathbb{P}^1,$
$\mu\left(\mathcal{F}_{\omega_{f/g}},C_{[\alpha:\beta]}\right)=I_0(f,g)+\mu(C_{[\alpha:\beta]}).$
\end{corollary}
\begin{proof}
By \cite[Proposition 4.1]{FS-GB-SM2},
\begin{equation}\label{ec2}
GSV(\mathcal{F}_{\omega _{f/g}},C_{[\alpha:\beta]})=\mu\left(\mathcal{F}_{\omega_{f/g}}, C_{[\alpha:\beta]}\right) - \mu(C_{[\alpha:\beta]}). 
\end{equation} 
As in Proposition \ref{Tju}, we have 
$GSV(\mathcal{F}_{\omega_{f/g}}, C_{[\alpha:\beta]})=I_0(f,g)$. Substituting this equality into \eqref{ec2} completes the proof.
\end{proof}
From Proposition \ref{Tju} and Corollary \ref{mu}, we obtain
\begin{corollary}\label{coro1} Under the assumptions of Proposition \ref{Tju}, the following identity holds:
    \[
    \mu(\mathcal{F}_{\omega _{f/g}}, C_{[\alpha:\beta]}) - \tau(\mathcal{F}_{\omega_{f/g}}, C_{[\alpha:\beta]})= \mu(C_{[\alpha:\beta]})- \tau(C_{[\alpha:\beta]}). 
    \]
    In particular
    $\mu(\mathcal{F}_{\omega _{f/g}}, C_{[\alpha:\beta]}) = \tau_{[\alpha:\beta]}(f,g)$ if and only if $C_{[\alpha:\beta]}$ is quasi-homogeneous by Saito \cite{Saito}. 
\end{corollary}

Similarly to the Milnor number of a generic member of a pencil,
the Tjurina number of a pair with respect to a fiber also exhibits a generic value. Indeed,  by \cite[Theorem 2.6]{G-L-S} the Tjurina number defines an upper semicontinuous function on the family of reduced curves $C_{[\alpha:\beta]}$. Hence, there exists a finite subset
$\mathcal B_{\textcolor{blue}{\tau}}(f,g)\subset\mathbb P^1$
such that $\tau(C_{[\alpha:\beta]})=\tau_{\mathrm{gen}}$
for $[\alpha:\beta]\notin\mathcal B_\tau(f,g)$, whereas
$\tau(C_{[\alpha:\beta]})>\tau_{\mathrm{gen}}$
for $[\alpha:\beta]\in\mathcal B_\tau(f,g)$, where
$\tau_{\mathrm{gen}}
:=\min_{[\alpha:\beta]\in\mathbb P^1}\tau(C_{[\alpha:\beta]})$
is the \textit{generic Tjurina number of the pencil
$\mathcal P_{f,g}$}. Now, 
from Proposition \ref{Tju} we get $\tau_{[\alpha:\beta]}(f,g)=I_0(f,g)+\tau\left(C_{[\alpha:\beta]}\right)$, hence
we can define the \textit{generic Tjurina number} of the pair $(f,g)$ as 
\begin{equation}\label{tjurina_geno}
    \tau_{gen}(f,g):=\tau_{[\alpha:\beta]}(f,g)\,\,\,\,\,\text{where}\,\,\,\,[\alpha:\beta]\in\mathbb{P}^1\setminus\B_{\textcolor{blue}{\tau}}(f,g).
\end{equation}
 Hence $\tau_{gen}(f,g)=I_0(f,g)+\tau_{gen}$, for $[\alpha:\beta]\in\mathbb{P}^1\setminus\B_{\tau}(f,g)$.
\par The following example compares two pairs whose corresponding generators are topologically
equivalent.
\begin{example}\label{ex1}
Let $f=y^3-x^7$ and $g=y^4-x^9$ in $\mathcal{O}_2$. Define $f_1=f+x^5y$ and $g_1=g-4yx^7-2y^2x^5+x^{10}$. It follows from \cite[p. 218]{G-L-S} that $f$ and $f_1$ are topologically equivalent. Likewise, $g$ and $g_1$ are topologically equivalent.   
 
Using Singular \cite{Singular}, we obtain 
$\tau_{gen}(f,g)=39$, whereas $\tau_{gen}(f_1,g_1)=38.$

\end{example}
\begin{remark}
    Example \ref{ex1} shows that the generic Tjurina number of a pair $(f,g)$ of holomorphic functions is not determined by the individual topological types of $f$ and $g$.
\end{remark}


\section{Tjurina number of a foliation induced by a pencil along its bifurcation and generic fibers}\label{sec: pencil_Tjurina}

The aim of this section is to establish a formula relating the Tjurina numbers of the foliation $\F_{\omega _{f/g}}$
induced by the pair $(f,g)$ along the zero and pole divisors of a specified family of balanced divisors of separatrices for $\F_{\omega_{f/g}}$. We use the construction in \cite[ Lemma 4.3]{FFMR}, allowing the distinguished generators themselves to be bifurcation fibers.

\par Let $\B_\mu(f,g)$ be the \textit{Milnor bifurcation set} of the pencil $\mathcal{P}_{f,g}$, namely
\[\B_\mu(f,g):=\{[\alpha:\beta]\in\mathbb{P}^1:\mu(C_{[\alpha:\beta]})>\mu_{gen}\},\qquad\mu_{gen}:=\min_{[\alpha:\beta]\in\mathbb{P}^1}\mu(C_{[\alpha:\beta]}).\]
Note that this finite set need not coincide with the set $\B_{\tau}(f,g)$ introduced in Section \ref{sec: Tjurina-pair-fiber}. In fact, for $f=y^3-x^7$ and $g=x^5y+y^4$, we have $\B_{\mu}(f,g)=\{[0:1]\}$ and 
${B}_{\tau}(f,g)=\{[1:0],[0:1]\}.$

In what follows, a \textit{generic fiber} means a member outside $\B_\mu(f,g)$. Whenever the value $\tau_{gen}$ is required, the fiber is also chosen outside $\B_{\tau}(f,g)$.

\par Before proceeding, we fix our notation and state some key results that will be used in this section.

Given pairwise relatively prime reduced nonunit germs
$\zeta_1,\ldots,\zeta_k\in\mathbb C\{x,y\}$, not necessarily
irreducible, we define $\Theta_{\zeta_1,\ldots,\zeta_k}\in\mathbb{Z}$
(see \cite[Theorem 3.4]{H-H}) by
\begin{equation}\label{tuj_eq}
    \Theta_{\zeta_1,\ldots,\zeta_k}
    :=\tau(\zeta_1\cdots\zeta_k)
    -\sum_{j=1}^{k}\tau(\zeta_j)
    -\sum_{1\leq i<j\leq k}I_0(\zeta_i,\zeta_j).
\end{equation}
If $k=1$ and $\zeta_1\in\C\{x,y\}$ is irreducible, we adopt
the convention that $\Theta_{\zeta_1}=0$.

\par Let $\F$ be a foliation germ at $(\C^2,0)$, and let
$C=\{\zeta=0\}$ be a reduced curve germ invariant by $\F$,
where $\zeta\in\C\{x,y\}$ is a reduced defining equation.
Throughout this paper, we use the notation  $\tau(\F,C)$
and $\tau(\F,\zeta)$ interchangeably to denote the number
defined in \eqref{def:Tj-fol-curve}.
We will use the following identity, proved in
\cite[Proposition 6.2]{FS-GB-SM1}:
\begin{equation}\label{eq_gsv}
\tau(\F,\zeta)=\tau(\zeta)+GSV(\F,\zeta).
\end{equation}

From now on, we  consider the foliation $\F_{\omega_{f/g}}$ defined by the $1$-form $\omega_{f/g}=gdf-fdg$, and assume that $\F_{\omega_{f/g}}$ has an isolated singularity at the origin. Observe that $\F_{\omega_{f/g}}$ is the natural foliation associated with the pencil $\mathcal{P}_{f,g}:\alpha f+\beta g$,  generated by $f$ and $g$. Let $h_1,\ldots, h_r$ denote the {\it bifurcation fibers} of this pencil different from $f$ and $g$. They correspond  to the elements indexed by $\B_\mu(f,g)^{*}=\B_\mu(f,g)\setminus\{[1:0],[0:1]\}$. By \cite[Lemma 4.3]{FFMR}, there exist generic fibers $\varphi_1,\ldots ,\varphi_r$ of $\mathcal{P}_{f,g}$ such that the foliation $\F_{\omega_{f/g}}$ admits a primitive balanced divisor of separatrices of the form: \[\D=(f)+(g)+(h_1)+ \cdots +(h_r)-(\varphi_1)- \cdots -(\varphi_r),\]
where $h_i=\alpha_i f+ \beta_i g$,  
with $\alpha_i,\beta_i\in\C$ for $1\leq i\leq r$, and
$\varphi_j=\gamma_j f+\delta_j g$, 
$\gamma_j,\delta_j\in\C$ \; for  \; $1\leq j\leq r$.

Since $h_i$ and $h_j$ are distinct fibers, neither of which is  $f$ or $g$,
we have $\beta_i\neq0$ and
$\beta_i\alpha_j-\beta_j\alpha_i\neq0$. Hence
\[
\begin{aligned}
I_0(h_i,h_j)
&=I_0(h_i,\beta_i h_j-\beta_j h_i)
=I_0\bigl(h_i,(\beta_i\alpha_j-\beta_j\alpha_i)f\bigr)\\
&=I_0(\alpha_i f+\beta_i g,f)=I_0(f,g).
\end{aligned}
\]
Here we use the invariance of intersection multiplicity under adding
a multiple of one generator to the other and under multiplying a generator
by a nonzero constant. The same argument applies to the generic fibers
$\varphi_1,\ldots,\varphi_r$, which are chosen to be pairwise distinct and different 
from $f$ and $g$. Therefore,
\begin{equation}\label{eq_sum1}
\sum_{1\leq i<j\leq r}I_0(h_i,h_j)
=
\sum_{1\leq i<j\leq r}I_0(\varphi_i,\varphi_j)
=
\frac{r(r-1)}{2}I_0(f,g),
\qquad r\geq1,
\end{equation}
and
\begin{equation}\label{eq_intersection}
I_0(f,g)
=I_0(f,h_i)
=I_0(f,\varphi_i)
=I_0(g,h_i)
=I_0(g,\varphi_i),
\qquad 1\leq i\leq r.
\end{equation}

We obtain the following lemma, which is stated for an arbitrary finite union of distinct members of the pencil.

\begin{lemma}\label{lemma_tjurina}
Let $\F_{\omega_{f/g}}$ be the foliation germ on
$(\mathbb{C}^2,0)$ induced by $\omega_{f/g}$, and let
$\zeta_1,\ldots,\zeta_k$, with $k\geq 1$, be pairwise
distinct members of $\mathcal{P}_{f,g}$.
 Then
\[\tau(\F_{\omega_{f/g}},\zeta_1\cdots \zeta_k)=\tau(\zeta_1\cdots \zeta_k)-k(k-2)I_0(f,g)\]
and 
\[\mu(\F_{\omega_{f/g}},\zeta_1\cdots \zeta_k)=\mu(\zeta_1\cdots \zeta_k)-k(k-2)I_0(f,g).\] 
\end{lemma}
\begin{proof} 
Applying \eqref{eq_gsv} to $\F_{\omega_{f/g}}$ and $C:\zeta_1\cdots \zeta_k=0$, we obtain
\begin{eqnarray}\label{lemma1}
\tau(\F_{\omega_{f/g}},C)=\tau(C)+GSV(\F_{\omega_{f/g}},C).
\end{eqnarray}
Now, using the branch formula for GSV \cite[p. 29]{Brunella-book},  \cite[Proposition 4.8]{FFMR}  and \eqref{eq_sum1}, we obtain
\begin{eqnarray}
    GSV(\F_{\omega_{f/g}},C)&=&\sum_{i=1}^{k}GSV(\F_{\omega_{f/g}},\zeta_i)-2\sum_{1\leq i<j \leq k}I_0(\zeta_i,\zeta_j) \nonumber\\\nonumber
    &=&kI_0(f,g)-k(k-1)I_0(f,g)\\
    &=& -k(k-2) I_0(f,g). \label{eq_gsv3}
\end{eqnarray}
Substituting \eqref{eq_gsv3} into \eqref{lemma1}, we obtain 
$\tau(\F_{\omega_{f/g}},\zeta_1\cdots \zeta_k)=\tau(\zeta_1\cdots \zeta_k)-k(k-2)I_0(f,g).$
On the other hand, using \cite[Proposition 4.1]{FS-GB-SM2}, we have that 
\[\mu(\F_{\omega_{f/g}},C)=\mu(C)+GSV(\F_{\omega_{f/g}},C).\]
The proof ends by substituting equation \eqref{eq_gsv3} into the previous formula. 
\end{proof}

\begin{corollary}\label{coroo}
Let $\F_{\omega_{f/g}}$ be a germ of a foliation at $(\C^2,0)$ induced by $\omega_{f/g}$ and let $h_1,\ldots,h_r$ denote the bifurcation fibers of $\mathcal{P}_{f,g}$ different from $f$ and $g$, where $r\geq 1$.  Then
\[\tau(\F_{\omega_{f/g}},h_1\cdots h_r)=\sum_{j=1}^{r}\tau(h_j)-\frac{r(r-3)}{2}I_0(f,g)+\Theta_{h_1,\ldots, h_r}.\]
\end{corollary}
\begin{proof}
Using Lemma \ref{lemma_tjurina}, \eqref{tuj_eq} and \eqref{eq_sum1}, we have
\[\begin{array}{lll}
\tau(\F_{\omega_{f/g}},h_1\cdots h_r)&=&\displaystyle\sum_{j=1}^{r}\tau(h_j)+\sum_{1\le i<j\le r} I_0(h_i,h_j)+\Theta_{h_1,\ldots, h_r}-r(r-2)I_0(f,g)\\
&=&\displaystyle\sum_{j=1}^{r}\tau(h_j)+\frac{r(r-1)}{2} I_0(f,g)+\Theta_{h_1,\ldots ,h_r}-r(r-2)I_0(f,g)\\
&=& \displaystyle\sum_{j=1}^{r}\tau(h_j)-\frac{r(r-3)}{2}I_0(f,g)+\Theta_{h_1,\ldots, h_r}.
\end{array}
\]
Hence, the corollary follows.
\end{proof}

\begin{lemma}\label{lemma_tjurina par}
Let $f,g\in\mathcal{O}_2$ be relatively
prime, and let $h_1,\ldots,h_r$, with
$r\geq 1$, be the bifurcation fibers of
$\mathcal{P}_{f,g}$ other than $f$ and
$g$. Let $\varphi_1,\ldots,\varphi_r$ be
pairwise distinct generic fibers of
$\mathcal{P}_{f,g}$. Then

\begin{equation*}
\begin{aligned}
\tau(fgh_1\cdots h_r)-\tau(\varphi_1\cdots\varphi_r)
&=\tau(fg)
+\sum_{j=1}^{r}\bigl(\tau(h_j)-\tau(\varphi_j)\bigr)
+2rI_0(f,g)\\
& 
+\Theta_{fg,h_1\cdots h_r}
+\Theta_{h_1,\ldots,h_r}
-\Theta_{\varphi_1,\ldots,\varphi_r}.
\end{aligned}
\end{equation*}
\end{lemma}
\begin{proof}
Applying \eqref{tuj_eq} to
$fg$ and $h_1\cdots h_r$, we obtain
\begin{equation}\label{eccua23}
\begin{aligned}
\tau(fgh_1\cdots h_r)
=&\, \tau(fg)+\tau(h_1\cdots h_r)\\
+&\, I_0(fg,h_1\cdots h_r)
+\Theta_{fg,h_1\cdots h_r}.
\end{aligned}
\end{equation}
Applying \eqref{tuj_eq} once more, this time to
$h_1,\ldots,h_r$, we get
\[
\tau(h_1\cdots h_r)
=\sum_{i=1}^{r}\tau(h_i)
+\sum_{1\leq i<j\leq r}I_0(h_i,h_j)
+\Theta_{h_1,\ldots,h_r}.
\]
Substituting this identity into \eqref{eccua23} and using
the additivity of intersection multiplicity, we obtain
\begin{equation}\label{eccua24}
\begin{aligned}
\tau(fgh_1\cdots h_r)
=&\, \tau(fg)+\sum_{i=1}^{r}\tau(h_i)
+\sum_{1\leq i<j\leq r}I_0(h_i,h_j)\\
+& \,I_0(f,h_1\cdots h_r)+I_0(g,h_1\cdots h_r)\\
+&\, \Theta_{h_1,\ldots,h_r}
+\Theta_{fg,h_1\cdots h_r}.
\end{aligned}
\end{equation}

\noindent
By the additivity of intersection multiplicity
and \eqref{eq_intersection}, we have

\begin{equation*}
I_0(f,h_1\cdots h_r)
=\sum_{i=1}^{r}I_0(f,h_i)
=\sum_{i=1}^{r}I_0(f,g)
=rI_0(f,g),
\end{equation*}
and
\begin{equation*}
I_0(g,h_1\cdots h_r)
=\sum_{i=1}^{r}I_0(g,h_i)
=\sum_{i=1}^{r}I_0(g,f)
=rI_0(f,g).
\end{equation*}
Using \eqref{eq_sum1} and substituting these equalities
into \eqref{eccua24}, we obtain
\begin{equation}\label{ecc26}
\begin{aligned}
\tau(fgh_1\cdots h_r)
&=\, \tau(fg)+\sum_{i=1}^{r}\tau(h_i)
+\frac{r(r+3)}{2}\,I_0(f,g)\\
&+\, 
\Theta_{fg,h_1\cdots h_r}+ \Theta_{h_1,\ldots,h_r}.
\end{aligned}
\end{equation}
On the other hand, applying \eqref{tuj_eq} to
$\varphi_1,\ldots,\varphi_r$, we obtain
\begin{equation}\label{ecua247}
\begin{aligned}
\tau(\varphi_1\cdots\varphi_r)
={}&\sum_{i=1}^{r}\tau(\varphi_i)
+\sum_{1\leq i<j\leq r}I_0(\varphi_i,\varphi_j)
+\Theta_{\varphi_1,\ldots,\varphi_r}\\
={}&\sum_{i=1}^{r}\tau(\varphi_i)
+\frac{r(r-1)}{2}\,I_0(f,g)
+\Theta_{\varphi_1,\ldots,\varphi_r}.
\end{aligned}
\end{equation}
Subtracting \eqref{ecua247} from \eqref{ecc26}, we derive
\begin{eqnarray*}
\tau(fgh_1\cdots h_r)-\tau(\varphi_1\cdots\varphi_r)&=&\tau(fg)
+\sum_{j=1}^{r}\bigl(\tau(h_j)-\tau(\varphi_j)\bigr)
+2rI_0(f,g)+\Theta_{fg,h_1\cdots h_r}
\\
&+&\Theta_{h_1,\ldots,h_r}-\Theta_{\varphi_1,\ldots,\varphi_r}
\end{eqnarray*}

\noindent which is the desired formula.
\end{proof}

With all above ingredients we can state the main result of this section.

\begin{theorem}\label{teo2}
Let $f,g\in\mathcal{O}_2$ be relatively
prime, and let $h_1,\ldots,h_r$, with
$r\geq 1$, be the bifurcation fibers of
$\mathcal{P}_{f,g}$ other than $f$ and
$g$. Let $\varphi_1,\ldots,\varphi_r$ be
pairwise distinct generic fibers of
$\mathcal{P}_{f,g}$, and let
$\mathcal{F}_{\omega_{f/g}}$ denote the
foliation germ at $(\mathbb{C}^2,0)$
induced by $\omega_{f/g}$. Then

\begin{eqnarray*}
    \tau(\F_{\omega_{f/g}},fgh_1\cdots h_r)-\tau(\F_{\omega_{f/g}},\varphi_1\cdots\varphi_r)&=&\tau(fg)+\sum_{j=1}^ r (\tau(h_j)-\tau(\varphi_j))-2r I_0(f,g)\\
    & +& \Theta_{fg,h_1\cdots h_r}
    +\Theta_{h_1, \ldots,h_r}-\Theta_{\varphi_1,\ldots,\varphi_r}.
\end{eqnarray*} 
\end{theorem}

\begin{proof}

Applying Lemma \ref{lemma_tjurina} with
$\zeta_1=f$, $\zeta_2=g$, and $\zeta_{i+2}=h_i$
for $i=1,\ldots,r$, we obtain
\begin{equation}\label{equa1}
    \tau(\F_{\omega_{f/g}},fg h_1\cdots h_r)
    =\tau(fg h_1\cdots h_r)-r(r+2)I_0(f,g).
\end{equation}
Similarly, taking $\zeta_i=\varphi_i$ for
$i=1,\ldots,r$, we deduce
\begin{equation}\label{equa2}
    \tau(\F_{\omega_{f/g}},\varphi_1\cdots \varphi_r)
    =\tau(\varphi_1\cdots \varphi_r)-r(r-2)I_0(f,g).
\end{equation}
Subtracting these equalities, we get
\[
\tau(\F_{\omega_{f/g}},fg h_1\cdots h_r)
-\tau(\F_{\omega_{f/g}},\varphi_1\cdots\varphi_r)=\tau(fg h_1\cdots h_r)-\tau(\varphi_1\cdots\varphi_r)-4rI_0(f,g).
\]

The result now follows from Lemma \ref{lemma_tjurina par},
applied to the difference
$\tau(fg h_1\cdots h_r)-\tau(\varphi_1\cdots\varphi_r)$.
\end{proof}


\section{The Tjurina number of pairs: a bifurcation formula}\label{sec:biju}

\par Recall that $\mathcal{B}_{\mu}(f,g)$ is the Milnor bifurcation set
introduced in Section \ref{sec: pencil_Tjurina}, whereas
$\mathcal{B}_{\tau}(f,g)$
denotes the Tjurina jumping set.

\par In \cite[Theorem 4.6]{FFMR}, the authors prove a conjecture posed by A. Szawlowski
in \cite{Sz1} concerning pencils of germs of holomorphic
functions in two variables.

\begin{theorem}[Bifurcation Formula \cite{FFMR}]\label{th bif}
Let $f,g \in \mathcal O_2$ be relatively prime and assume that both
vanish at the origin. Then
\[
\mu(f,g)=\mu(fg)+
\sum_{[\alpha:\beta]\in \mathcal{B}_\mu(f,g)^{*}}
(\mu(C_{[\alpha:\beta]})-\mu_{gen}),
\]
where $\mathcal{B}_{\mu}(f,g)^{*}
=\mathcal{B}_{\mu}(f,g)\backslash \{[1:0],[0:1]\}$.
\end{theorem}

\par This result
motivates the present section, in which
we introduce the Tjurina number of a pair of holomorphic functions
and establish a bifurcation formula for this invariant.
To this end, we first introduce
some preliminary concepts and results.

\begin{definition}\label{def:Tjurina-general}

Let $\mathcal{F}:\omega=0$ be a holomorphic foliation, where
$\omega=P(x,y)dx+Q(x,y)dy$, and let
$\mathcal{D}=\mathcal{D}_0-\mathcal{D}_{\infty}$ be a primitive
balanced divisor of separatrices
for  $\mathcal{F}$.

We define the
{\it Tjurina number of $\mathcal{F}$} with respect
to $\mathcal{D}$ by
\begin{equation}\label{defi_Tjurina}
    \tau(\mathcal{F},\mathcal{D})
    :=\tau(\mathcal{F},\mathcal{D}_0)
    -\tau(\mathcal{F},\mathcal{D}_{\infty})
    +I_0(\mathcal{D}_0,\mathcal{D}_{\infty}).
\end{equation}

\end{definition}

\noindent Observe that, when $\mathcal D$ is an effective divisor,
that is, when $\D_{\infty}=\emptyset$,
Definition \ref{def:Tjurina-general} coincides with Definition \eqref{def:Tj-fol-curve}.

\noindent
Since $\tau(C)$ is an analytic invariant of $C$ and
$GSV(\F,C)$ is an analytic invariant of the pair $(\F,C)$, it follows from \eqref{eq_gsv}
that
$\tau(\F,C)=\tau(C)+GSV(\F,C)$
is an analytic invariant of the pair $(\F,C)$.
Consequently, $\tau(\F,\D_0)$ and
$\tau(\F,\D_{\infty})$ are analytic invariants of the pairs
$(\F,\D_0)$ and $(\F,\D_{\infty})$, respectively.

\par It follows that, for $\D=\D_0-\D_{\infty}$,
$\tau(\F,\D)$ is an analytic invariant of the pair $(\F,\D)$.

In this section, we study the notion $\tau(\F,\D)$ when
$\F=\F_{\omega_{f/g}}$ and the zero divisor of $\mathcal D$ is
\[
\mathcal{D}_0=(f)+(g)+(h_1)+\cdots+(h_r),
\]
as introduced in Section~3, whereas the pole divisor
is given by
\[
\mathcal{D}_{\infty}=(\varphi_1)+\cdots+(\varphi_r),
\]
which varies.

Let $U\subset\C^r$ be the set of tuples
$\mathbf{t}=(t_1,\ldots,t_r)$ for which the fibers
\[
\varphi_j=f+t_jg,\qquad j=1,\ldots,r,
\]
are pairwise distinct, different from $f$, and correspond to points $[1:t_j]$ outside $\mathcal B_\mu(f,g)\cup\mathcal B_\tau(f,g)$,
so that each $\varphi_j$ is a Milnor generic fiber
and satisfies $\tau(\varphi_j)=\tau_{\mathrm{gen}}$.
The fiber $g$, corresponding to $[0:1]$, is already excluded
by this affine parametrization. 
Since $\mathcal B_\mu(f,g)\cup\mathcal B_\tau(f,g)$
is finite, $U$ is the complement of a finite union
of hyperplanes. Indeed, the excluded loci are
$t_i=t_j$ for $i<j$, $t_j=0$, and $t_j=a$, where
$a$ ranges over the parameters such that
$[1:a]\in\mathcal B_\mu(f,g)\cup\mathcal B_\tau(f,g)$.
Therefore, $U$ is a nonempty Zariski open subset of $\C^r$.
Define
\[
\Phi_{\mathbf{t}}
:=\prod_{j=1}^r(f+t_jg),
\qquad \mathcal{D}_{\infty,\mathbf{t}}
:=(\Phi_{\mathbf{t}}),\qquad
\mathcal{D}_{\mathbf{t}}
:=\mathcal{D}_0-\mathcal{D}_{\infty,\mathbf{t}}.
\]
For every $\mathbf{t}\in U$, the divisor
$\mathcal{D}_{\mathbf{t}}$ is a primitive balanced divisor
of $\F_{\omega_{f/g}}$ adapted to $fg=0$.
When $r=0$, $U$ consists of a single point,
$\Phi_{\mathbf{t}}=1$, and $\mathcal{D}_{\mathbf{t}}=(fg)$.

\begin{lemma}[Generic constancy]\label{l:independent}
Let $\F_{\omega_{f/g}}$ be the holomorphic foliation defined by $\omega_{f/g}=gdf-fdg$, where $f,g\in\mathcal{O}_2$ are relatively prime and $\F_{\omega_{f/g}}$ has an isolated singularity at $0\in\C^2$. There exists a nonempty Zariski open subset $U^{\circ}\subset U$ on which $\tau(\F_{\omega_{f/g}},\mathcal{D}_{\mathbf{t}})$ is constant. This constant is the maximum of this number on $U$. Moreover, $U^{\circ}$ may be chosen so that $\tau(\varphi_j)=\tau_{gen}$ for every $j$.
\end{lemma}
\begin{proof}
Write $P=gf_x-fg_x$, $Q=gf_y-fg_y$, and consider the fixed algebra
\[
A=\frac{\mathcal{O}_2}{(P,Q)},\qquad\dim_{\C}A=\mu(f,g).
\]
For $a\in\mathcal{O}_2$, denote its class in $A$ by $\overline{a}$.
For each  tuple $\mathbf{t}\in U$,
define the multiplication map
\[
\begin{array}{rcl}
M_{\mathbf{t}}:A&\longrightarrow&A,\\
\overline{a}&\longmapsto&\overline{\Phi_{\mathbf{t}}a}.
\end{array}
\]
This is a well-defined $\C$-linear endomorphism of $A$, since multiplication by $\overline{\Phi_{\mathbf{t}}}$ preserves $A$.
Its image is
$\operatorname{Im}M_{\mathbf{t}}=\overline{\Phi_{\mathbf{t}}}A$.
By the third isomorphism theorem, we have
\[
\frac{\mathcal{O}_2}{(P,Q,\Phi_{\mathbf{t}})}
\simeq
\frac{A}{\operatorname{Im}M_{\mathbf{t}}},
\]
and taking dimensions gives
\begin{equation}\label{eq:multiplication-rank}
\tau(\F_{\omega_{f/g}},\D_{\infty,\mathbf{t}})
=\mu(f,g)-\operatorname{rank}M_{\mathbf{t}}.
\end{equation}

To see the parameter dependence explicitly, let $M_f$ and $M_g$
denote multiplication by $\overline{f}$ and $\overline{g}$ in $A$.
For a single pole, $\Phi_t=f+tg$, so
\[
M_t(\overline{a})=\overline{fa}+t\overline{ga},
\qquad M_t=M_f+tM_g.
\]
For $r$ poles, these multiplication maps commute and
\[
M_{\mathbf{t}}=(M_f+t_1M_g)\cdots(M_f+t_rM_g).
\]

Fix a basis of $A$, which does not depend on $\mathbf{t}$.
With respect to this basis, the matrices of $M_f$ and $M_g$
are constant. Hence, the entries of the matrix of
$M_{\mathbf{t}}$ are polynomials in
$t_1,\ldots,t_r$.

Let $R$ be the maximum rank attained on $U$.
If $R>0$, the locus where at least one minor of size $R$
is nonzero is a nonempty Zariski open subset of $U$,
and the rank equals $R$ there.
If $R=0$, the rank is zero throughout $U$.
Thus \eqref{eq:multiplication-rank} has the minimum value
$\mu(f,g)-R$ on a nonempty open subset.

The intersection term $I_0(\D_0,\D_{\infty,\mathbf{t}})$
is constant, since
$I_0(\mathcal{D}_0,\D_{\infty,\mathbf{t}})
=r(r+2)I_0(f,g)$ by \eqref{eq_sum1} and
\eqref{eq_intersection}. If
\[
c=\tau(\F_{\omega_{f/g}},\mathcal{D}_0)
+r(r+2)I_0(f,g),
\]
definition \eqref{defi_Tjurina} yields
\[
\begin{aligned}
\tau(\F_{\omega_{f/g}},\mathcal{D}_{\mathbf{t}})
&=\tau(\F_{\omega_{f/g}},\D_0)
-\tau(\F_{\omega_{f/g}},\D_{\infty,\mathbf{t}})
+I_0(\D_0,\D_{\infty,\mathbf{t}})\\
&=c-\mu(f,g)+\operatorname{rank}M_{\mathbf{t}}.
\end{aligned}
\]
Hence this number attains its maximum on the maximum-rank
locus, which is a nonempty Zariski open subset of $U$.

Since $U$ is a nonempty Zariski open subset
of $\C^r$, it is irreducible. Thus the maximum-rank locus is
a nonempty Zariski open subset of $U$. Moreover, by the
definition of $U$, all the fibers $\varphi_j$ are generic,
and hence $\tau(\varphi_j)=\tau_{gen}$ for every
$j=1,\ldots,r$. We can therefore take the maximum-rank
locus as the required subset $U^{\circ}$.

When $r=0$, the product is empty,
so $\Phi_{\mathbf{t}}=1$ and $M_{\mathbf{t}}$ is the
identity map of $A$. The assertion follows immediately.
\end{proof}

Lemma \ref{l:independent} allows  us to introduce the following definition.
\begin{definition}\label{def:Tjurina-pair}
Let $f,g \in \mathcal O_2$ be relatively prime germs,  both vanishing at the origin.
The {\it Tjurina number of the pair} $(f,g)$ is 
defined by
\begin{equation}\label{eq:Tjurina-pair}
\tau(f,g):=\max_{\mathbf{t}\in U}\tau(\F_{\omega_{f/g}},\mathcal{D}_{\mathbf{t}})
=\tau(\F_{\omega_{f/g}},\mathcal{D}_{\mathbf{t}}),\qquad\mathbf{t}\in U^{\circ}.
\end{equation}
\end{definition}

This number is an analytic invariant of the meromorphic germ $f/g$, together with its distinguished zero and pole fibers. Indeed, if $(f',g')=(u(f\circ\psi),u(g\circ\psi))$, where $\psi$ is a biholomorphic change of source coordinates and $u$ is a unit, then
\[
g'df'-f'dg'=u^2\psi^*(gdf-fdg).
\]
The pencil curves, the bifurcation set, the specified family of divisors, and all the local colengths involved in the definition are preserved under this transformation. We emphasize that $\tau_{gen}(f,g)$ in \eqref{tjurina_geno} continues to denote the invariant associated with a single generic fiber.

\begin{remark}
\label{rem:essential}
The restriction to the family $U$ in Definition \ref{def:Tjurina-pair} is essential. Consider $f=x$ and $g=y$, so that $\F_{\omega_{f/g}}$ is defined by $\omega_{f/g}=ydx-xdy$. The divisors
\[
\mathcal{\D}=(x)+(y),\qquad
\widetilde{\mathcal{D}}=(x)+(y)+(x-y)-(x-2y)
\]
are both primitive balanced divisors adapted to $fg=0$. For $\D$, we have
\[
\tau(\F_{\omega_{f/g}},\D)
=\dim_{\C}\frac{\mathcal{O}_2}{(x,y,xy)}=1.
\]
whereas
\[
\tau(\F_{\omega_{f/g}},\widetilde{\mathcal{D}})=\tau(\F_{\omega_{f/g}},\widetilde{\mathcal{D}}_0)-\tau(\F_{\omega_{f/g}},\widetilde{\mathcal{D}}_{\infty})+I_0(\widetilde{\mathcal{D}}_0,\widetilde{\mathcal{D}}_{\infty})
=1-1+I_0\bigl(xy(x-y),x-2y\bigr)=3.
\]
Thus, the value associated with an arbitrary balanced divisor depends on the choice of divisor.\\
For our definition, every member of the pencil $\alpha x+\beta y=0$ is smooth, so there are no additional bifurcation fibers and $r=0$. Hence, $U$ consists of a single point, $\Phi_{\mathbf{t}}=1$, and the prescribed divisor is $\mathcal{D}_{\mathbf{t}}=(x)+(y)$. Therefore,
\[
\tau(x,y)
=\max_{\mathbf{t}\in U}
\tau(\F_{\omega_{f/g}},\mathcal{D}_{\mathbf{t}})
=\dim_{\C}\frac{\mathcal{O}_2}{(x,y,xy)}
=1.
\]
The divisor $\widetilde{\mathcal{D}}$ does not belong to this family, since its positive part contains the additional fiber $x-y$, which is not a bifurcation fiber.\\
In fact, the values obtained by considering all primitive balanced divisors adapted to \(fg=0\) are unbounded.
Indeed, for pairwise distinct lines $x,y,\ell_1,\ldots,\ell_s,m_1,\ldots,m_s$, the divisor
\[
\mathcal{D}_s=(x)+(y)+\sum_{i=1}^{s}(\ell_i)
-\sum_{j=1}^{s}(m_j),\qquad s\geq1,
\]
is balanced and satisfies
\[
\tau(\F_{\omega_{f/g}},\mathcal{D}_s)
=1-1+
I_0\left(xy\prod_{i=1}^{s}\ell_i,\prod_{j=1}^{s}m_j\right)
=s(s+2).
\]
\end{remark}

\begin{proposition}[Bifurcation formula for the Tjurina number of pairs]
\label{bifuti}

Let $f,g\in\mathcal{O}_2$ be relatively
prime, and let $h_1,\ldots,h_r$ be the
bifurcation fibers of
$\mathcal{P}_{f,g}$ other than $f$ and
$g$, where $r\geq 1$.
Let $\F_{\omega_{f/g}}$ denote the
foliation germ at $(\C^2,0)$ induced by
$\omega_{f/g}$.
Let 
$\mathbf{t}=(t_1,\ldots,t_r)\in
U^{\circ}$ and set
$\varphi_j=f+t_jg$ for $j=1,\ldots,r$.
The primitive
balanced divisor adapted to $fg=0$
corresponding to
$\varphi_1,\ldots,\varphi_r$ is

\[
\mathcal{D}_{\mathbf{t}}
=(f)+(g)+\sum_{j=1}^{r}(h_j)
-\sum_{j=1}^{r}(\varphi_j).
\]
 Then

\[
\begin{aligned}
\tau(f,g)
=& \,\tau(fg)
+\sum_{j=1}^{r}\bigl(\tau(h_j)-\tau_{gen}\bigr)
+r^2 I_0(f,g)\\
+& \, \Theta_{fg,h_1\cdots h_r}
+\Theta_{h_1,\ldots,h_r}
-\Theta_{\varphi_1,\ldots,\varphi_r}.
\end{aligned}
\]
\end{proposition}

\begin{proof}
The zero and pole divisors 
 of $\mathcal{D}_{\mathbf{t}}$
are given  by $fg h_1\cdots h_r$ and
$\varphi_1\cdots\varphi_r$, respectively.
Thus, by \eqref{eq_sum1} and \eqref{eq_intersection},
\[
I_0(fg h_1\cdots h_r,\varphi_1\cdots\varphi_r)
=r(r+2)I_0(f,g).
\]

Since $\mathbf{t}\in U^{\circ}$, we also have
\[
\tau(\F_{\omega_{f/g}},\mathcal{D}_{\mathbf{t}})
=\tau(f,g),
\qquad
\tau(\varphi_j)=\tau_{gen}
\quad (j=1,\ldots,r).
\]
Substituting these identities into \eqref{defi_Tjurina}
and applying Theorem~\ref{teo2} gives the claimed formula,
since the coefficient of $I_0(f,g)$ simplifies to
$r(r+2)-2r=r^2$.
\end{proof}
To illustrate Proposition~\ref{bifuti}, we consider the following
example, originally due to Szawlowski
(see \cite[p. 294]{Sz2} or \cite[Example~4.5]{FFMR}).

\begin{example}\label{ex_bifur}
Let
$f=x^3-3x^2y+y^3+y^5$ and
$g=y^3-3x^2y$, and write
$\F=\F_{\omega_{f/g}}$, where
$\omega_{f/g}=P\,dx+Q\,dy$ with
\[
\begin{aligned}
P&=-3x^4y+3x^2y^3+6xy^6,\\
Q&=3x^5-3x^3y^2-12x^2y^5+2y^7.
\end{aligned}
\]
The bifurcation fibers other than $f$
and $g$ are
\[
h_1=f-g,\qquad h_2=2f-g,\qquad
h_3=2f-3g.
\]
Take $\mathbf{t}_0=(1,2,3)$, so that
$\varphi_j=f+jg$ for $j=1,2,3$. Each
$\varphi_j$ has an ordinary triple
point at the origin, hence
$\tau(\varphi_j)=\tau_{gen}=4$.

Put $H=h_1h_2h_3$,
$\mathcal{D}_0=fgH$ and
$\Phi_{\mathbf{t}_0}
=\varphi_1\varphi_2\varphi_3$. The
primitive balanced divisor of separatrices for
$\F$ adapted to
$fg=0$ corresponding to
$\varphi_1,\varphi_2,\varphi_3$ is
$\mathcal{D}_{\mathbf{t}_0}
=\mathcal{D}_0-(\Phi_{\mathbf{t}_0})$.
Computations in Singular
\cite{Singular} give
\[
\begin{gathered}
\tau(fg)=24,\qquad
\tau(h_1)=8,\qquad
\tau(h_2)=\tau(h_3)=6,\qquad
I_0(f,g)=9,\\
\tau(H)=60,\qquad
\tau(\mathcal{D}_0)=168,\qquad
\tau(\Phi_{\mathbf{t}_0})=58,\\
\tau(\F,\mathcal{D}_0)=33,\qquad
\tau(\F,\Phi_{\mathbf{t}_0})=31.
\end{gathered}
\]
Consequently,
\[
\begin{aligned}
\Theta_{fg,H}
&=168-24-60-54=30,\\
\Theta_{h_1,h_2,h_3}
&=60-(8+6+6)-27=13,\\
\Theta_{\varphi_1,\varphi_2,\varphi_3}
&=58-3\cdot4-27=19.
\end{aligned}
\]
It remains to check that
$\mathbf{t}_0$ realizes the generic
maximum in the definition of
$\tau(f,g)$. Let
$A=\mathcal{O}_2/(P,Q)$ and
$\mathfrak{m}_A=(\overline{x},
\overline{y})$. Singular \cite{Singular} gives
$\dim_{\C}A=33$ and
$\mathfrak{m}_A^{10}
=\C\,\overline{xy^9}\neq0$.
Since $f,g\in(x,y)^3$, every product
$\Phi_{\mathbf{s}}$ of three members of
the pencil lies in $(x,y)^9$. Hence,
for every $\mathbf{s}\in U$,
\[
\operatorname{Im}M_{\mathbf{s}}
=\overline{\Phi_{\mathbf{s}}}A
\subseteq
\C\,\overline{\Phi_{\mathbf{s}}}
+\mathfrak{m}_A^{10},
\]
so $\operatorname{rank}M_{\mathbf{s}}
\leq2$. For $\mathbf{t}_0$ we have
$\operatorname{rank}M_{\mathbf{t}_0}
=\dim_{\C}A
-\tau(\F,\Phi_{\mathbf{t}_0})
=33-31=2$, so $\mathbf{t}_0$ attains
the maximal rank and realizes the
generic value defining $\tau(f,g)$.
By \eqref{defi_Tjurina},
\[
\tau(f,g)=33-31+3\cdot5\cdot9=137.
\]
The right-hand side of
Proposition~\ref{bifuti} equals
\[
24+(8-4)+(6-4)+(6-4)
+3^2\cdot9+30+13-19=137,
\]
which verifies
Proposition~\ref{bifuti}.
\end{example}

\begin{theorem}\label{coro:bif}
Let $f,g\in\mathcal{O}_2$ be relatively prime germs
vanishing at the origin, and let $h_1,\ldots,h_r$ be
the bifurcation fibers of $\mathcal{P}_{f,g}$ other
than $f$ and $g$, where $r\geq 0$.
Choose $\mathbf{t}=(t_1,\ldots,t_r)\in U^{\circ}$ and set
\[
\varphi_j=f+t_jg,\qquad
\Phi_{\mathbf{t}}=\prod_{j=1}^{r}\varphi_j,
\qquad
\mathcal{D}_0=fg h_1\cdots h_r.
\]
Let
$\mathcal{D}_{\mathbf{t}}
=\mathcal{D}_0-(\Phi_{\mathbf{t}})$
be the corresponding primitive balanced divisor of
separatrices for $\F_{\omega_{f/g}}$ adapted to $fg=0$.
Then
\[
\begin{aligned}
\tau(f,g)
&=\tau(\F_{\omega_{f/g}},\mathcal{D}_{\mathbf{t}})\\
&=\tau(\mathcal{D}_0)-\tau(\Phi_{\mathbf{t}})
+r(r-2)I_0(f,g).
\end{aligned}
\]
When $r=0$, we use the conventions
$\Phi_{\mathbf{t}}=1$ and $\tau(1)=0$.
\end{theorem}
\begin{proof}
Write $\F=\F_{\omega_{f/g}}$.
Since $\F$ is a generalized curve foliation,
\cite[Corollary 6.4]{FS-GB-SM1} applied to
$\mathcal{D}_{\mathbf{t}}
=\mathcal{D}_0-(\Phi_{\mathbf{t}})$,
with $\mathcal{D}_0=\mathcal{D}_0$ and
$\mathcal{D}_\infty=(\Phi_{\mathbf{t}})$, gives
\[
\tau(\F,\mathcal{D}_0)
=\tau(\mathcal{D}_0)
-I_0(\mathcal{D}_0,\Phi_{\mathbf{t}}).
\]
Thus, by \eqref{defi_Tjurina},
\[
\begin{aligned}
\tau(\F,\mathcal{D}_{\mathbf{t}})
&=\tau(\F,\mathcal{D}_0)
-\tau(\F,\Phi_{\mathbf{t}})
+I_0(\mathcal{D}_0,\Phi_{\mathbf{t}})\\
&=\tau(\mathcal{D}_0)-\tau(\F,\Phi_{\mathbf{t}}).
\end{aligned}
\]
On the other hand, \eqref{equa2} yields
$\tau(\F,\Phi_{\mathbf{t}})
=\tau(\Phi_{\mathbf{t}})-r(r-2)I_0(f,g)$,
with the same identity holding trivially when $r=0$.
Substituting this into the preceding equality gives
\[
\tau(\F,\mathcal{D}_{\mathbf{t}})
=\tau(\mathcal{D}_0)-\tau(\Phi_{\mathbf{t}})
+r(r-2)I_0(f,g).
\]
Finally, the choice $\mathbf{t}\in U^{\circ}$ ensures that
$\tau(\F,\mathcal{D}_{\mathbf{t}})=\tau(f,g)$,
which completes the proof.
\end{proof}

\begin{example}\label{ejemplo43}
Returning to Example \ref{ex_bifur}, we have $r=3$ and
\[
\tau(\mathcal{D}_0)=168,\qquad
\tau(\Phi_{\mathbf{t}_0})=58,\qquad
I_0(f,g)=9.
\]
Thus, Theorem \ref{coro:bif} gives
\[
\tau(f,g)=168-58+3(3-2)\cdot9=137,
\]
in agreement with the  computation in Example \ref{ex_bifur}.
Moreover,
\[
\mu(f,g)=33,\qquad
\tau_{gen}(f,g)=I_0(f,g)+\tau_{gen}=9+4=13.
\]
\end{example}
\begin{remark}
Example~\ref{ejemplo43} shows that
$\tau(f,g)$ can be strictly larger
than $\mu(f,g)$. Recall that
$\tau_{gen}(f,g)$ denotes the Tjurina
number of the pair with respect to a
single generic fiber (not to be
confused with $\tau_{gen}$, the Tjurina
number of a generic fiber). After \eqref{Mpair} and \eqref{tjurina_geno}, it always
satisfies
$\tau_{gen}(f,g)\leq\mu(f,g)$;
in the present example,
\[
\tau_{gen}(f,g)=13<\mu(f,g)=33
<\tau(f,g)=137.
\]
\end{remark}

\subsection{Characterization of semitame meromorphic function germs}

Following \cite{Bodin1,Sz2}, a
meromorphic function germ
$(f:g):(\C^2,0)\dashrightarrow\mathbb{P}^1$
is called \emph{semitame} if
its bifurcation
fibers, if any, are among $f$ and $g$,
where bifurcation is understood in the
Milnor sense.
The Milnor bifurcation formula of
\cite{FFMR}
shows that $(f:g)$
is semitame if and only if
$\mu(f,g)=\mu(fg)$.

\par We obtain the following characterization in terms of the
Tjurina number by combining Theorem \ref{coro:bif} with
the product bounds established in
\cite[Theorems 3.4 and 4.3]{H-H}: for relatively prime
reduced nonunit germs $\zeta_1,\zeta_2\in\mathcal{O}_2$, we have
\begin{equation}\label{eq:product-tjurina-bounds}
\tau(\zeta_1)+\tau(\zeta_2)+I_0(\zeta_1,\zeta_2)
\leq\tau(\zeta_1\zeta_2)
\leq\tau(\zeta_1)+\tau(\zeta_2)+2I_0(\zeta_1,\zeta_2)-1.
\end{equation}

\begin{corollary}\label{caracsemitame}
Let $f,g\in\mathcal{O}_2$ be relatively
prime germs
vanishing at the
origin,
and assume that every member of
$\mathcal{P}_{f,g}$,
including $f$ and $g$, is reduced.
Then the meromorphic
germ $(f:g)$ is semitame if and only if
$\tau(f,g)=\tau(fg)$.
\end{corollary}

\begin{proof}
Let $h_1,\ldots,h_r$ be the bifurcation fibers other
than $f$ and $g$. Semitameness is equivalent to $r=0$.
In this case, $\mathcal{D}_0=fg$ and there are no pole
fibers, so Theorem \ref{coro:bif} gives
$\tau(f,g)=\tau(fg)$.

Conversely, assume that $\tau(f,g)=\tau(fg)$.
Suppose, for a contradiction, that $r\geq1$.
Choose $\mathbf{t}\in U^{\circ}$ and write
\[
I=I_0(f,g),\qquad
\mathcal{D}_0=fg h_1\cdots h_r,\qquad
\Phi_{\mathbf{t}}=\varphi_1\cdots\varphi_r.
\]
Distinct pencil members have intersection multiplicity
$I$. In particular,
\[
I_0(fg,h_j)=2I,\qquad
I_0(h_i,h_j)=I_0(\varphi_i,\varphi_j)=I
\quad (i\neq j).
\]
Applying the lower bound in
\eqref{eq:product-tjurina-bounds} successively to
$fg,h_1,\ldots,h_r$, we obtain
\[
\tau(\mathcal{D}_0)
\geq
\tau(fg)+\sum_{j=1}^{r}\tau(h_j)
+\frac{r(r+3)}{2}I.
\]
Indeed, the intersections involving $fg$ contribute
$2rI$, while those among the $h_j$ contribute
$r(r-1)I/2$.
\par Similarly, applying the upper bound successively to
the $r$ pole fibers, and using
$\tau(\varphi_j)=\tau_{gen}$, gives
\[
\tau(\Phi_{\mathbf{t}})
\leq
r\tau_{gen}+r(r-1)I-r+1.
\]
Combining these inequalities with
Theorem \ref{coro:bif}, we find
\begin{equation}
\label{eq:tjurina-semitame-excess}
\tau(f,g)-\tau(fg)
\geq{}
\sum_{j=1}^{r}\bigl(\tau(h_j)-\tau_{gen}\bigr)
+\frac{r(r+1)}{2}I+r-1
>0.
\end{equation}

The strict inequality follows from
$\tau(h_j)\geq\tau_{gen}$, $I\geq1$, and $r\geq1$.
Since $\tau(f,g)=\tau(fg)$, this yields $0>0$,
a contradiction. Therefore, $r=0$ and $(f:g)$ is
semitame, proving the converse.
\end{proof}

\begin{remark}
Proposition \ref{Tju} gives
\[
\tau_{[1:0]}(f,g)=\tau(f)+I_0(f,g),
\qquad
\tau_{[0:1]}(f,g)=\tau(g)+I_0(f,g).
\]
Consequently, the identity
\begin{equation}\label{eq:corotame}
\tau_{[1:0]}(f,g)+\tau_{[0:1]}(f,g)
=\tau(f)+\tau(g)+2I_0(f,g)
\end{equation}
holds for every pair under consideration,
independently of semitameness.
\end{remark}

\begin{example}
Let $m,n$ be coprime integers with $2\leq m<n$, and set
\[
f=x^m+y^n,\qquad g=x^n+y^m.
\]
Their tangent cones are $x^m$ and $y^m$, respectively.
Hence $f$ and $g$ are relatively prime and
$I_0(f,g)=m^2.$
Moreover, both generators are quasihomogeneous, so
\begin{equation}\label{eq:quasihomogeneous-generators}
\tau(f)=\mu(f)=\tau(g)=\mu(g)=(m-1)(n-1).
\end{equation}
For $\alpha\beta\neq0$, the tangent cone of
$\alpha f+\beta g$ is
$\alpha x^m+\beta y^m$,
which is a product of $m$ distinct linear forms.
Thus every such member is an ordinary $m$-fold point,
with Milnor number $(m-1)^2$.
Since
$(m-1)(n-1)>(m-1)^2$,
the only bifurcation fibers are $f$ and $g$.
All pencil members are reduced, and $(f:g)$ is therefore
semitame. Corollary \ref{caracsemitame} yields
\[
\tau(f,g)=\tau(fg).
\]
A direct computation gives
\[
\tau(fg)
=(m+n)^2-n(n+2)
=m^2+2n(m-1).
\]
Equivalently, $\Theta_{f,g}=2(m-1)$. Hence
\[
\tau(f,g)=\tau(fg)=m^2+2n(m-1).
\]

Finally, Proposition \ref{Tju} gives
\[
\tau_{[1:0]}(f,g)=\tau_{[0:1]}(f,g)
=m^2+(m-1)(n-1),
\]
in agreement with \eqref{eq:corotame}, since
\[
\tau_{[1:0]}(f,g)+\tau_{[0:1]}(f,g)
=2m^2+2(m-1)(n-1)
=\tau(f)+\tau(g)+2I_0(f,g).
\]
Taking $m=2$ and $n=3$, we recover
\cite[Example 10]{Bodin2}.
\end{example}

\begin{corollary}\label{caracsemitame1}
Let $f,g\in\mathcal{O}_2$ be relatively prime nonunit germs.
Assume that every member of $\mathcal{P}_{f,g}$ is
reduced and that $(f:g)$ is semitame. Then
\begin{equation}\label{eq:corotame1}
\mu(f,g)-\tau(f,g)=\mu(fg)-\tau(fg).
\end{equation}
\end{corollary}

\begin{proof}
By Theorem \ref{th bif}
\begin{equation}\label{muu}
\mu(f,g)=\mu(fg).
\end{equation}
On the other hand, Corollary \ref{caracsemitame} gives
\begin{equation}\label{tiu}
\tau(f,g)=\tau(fg).
\end{equation}
Subtracting \eqref{tiu} from \eqref{muu} proves the result.
\end{proof}


\section{Teissier Lemma for the pair $(f,g)$}
\label{sec: Teissier}
Before establishing the results of this section, let us recall the definition of polar curves of foliations, a concept that was introduced in \cite{Corral1}.
\par Let $\omega=P(x,y)dx+Q(x,y)dy$ be a 1-form, where $P(x,y),Q(x,y)\in \mathcal{O}_2$.
If $\F: \omega=0$ is a singular holomorphic foliation then {\it the polar curve} of $\F$  at $(\C^2,p)$ with respect to a point $[a:b]$ of the complex projective line $\mathbb{P}^1$ is  the curve $\mathcal{P}^{\F}_{[a:b]}: aP(x,y)+bQ(x,y)=0$. 
Observe that when $\F:df=0$ is the hamiltonian foliation associated to a function $f$, the polar curve of $\F$ coincides with the classical polar curve of $f$ in the direction $[a:b]$ studied by Teissier \cite{Teissier_inv} and others.  According to the general results on equisingularity (see \cite{Zariski} and \cite{Teissier_polar}), there exists a Zariski open $U$ of the space $\mathbb P^1$ of projection directions such that for $[a:b]\in U$  the polar curves are all equisingular. Any polar curve corresponding to a direction in this set is called a {\it generic polar curve} of the foliation $\F$ and we will denote it by $\mathcal{P}^{\F}$. 
Let $B: h(x,y)=0$ be a separatrix  of a singular foliation $\F$. The {\it polar intersection number} of $\F$ with respect to $B$ is the intersection number $I_0({\mathcal P}^{\F},B)$.

In the context of complex hypersurfaces, Teissier proved the following formula relating the
polar intersection number $I_0(\mathcal{P}^{df},C)$ (see \cite[Section~4]{FS-GB-SM1}) and the Milnor
number $\mu(C)$; this result is known as Teissier's Lemma (see
\cite[Chapter II, Proposition~1.2]{Teissier}):

\begin{lemma}[Teissier's Lemma]\label{lemma_tei}
    Let $f\in\mathcal{O}_2$ be such that $C=\{f=0\}$ is a reduced plane curve germ at $0\in\C^2$. Then
\[I_{0}(\mathcal{P}^{df},C)=\mu(C)+\nu(f)-1,\]
where $\nu(f)$ denotes the algebraic multiplicity of $f$ at the origin.
\end{lemma}

For the foliation defined by $\omega_{f/g}=(gf_{x}-f g_{x})dx + (gf_{y}-f g_{y})dy$, its polar curve with respect to $[a:b]\in\mathbb{P}^1$ is given by

\[\mathcal{P}_{[a:b]}^{\omega_{f/g}}=(gf_{x}-f g_{x})a + (gf_{y}-f g_{y})b\,\,\,\text{for}\,\,\,[a:b] \in \mathbb{P}^{1}.\] Observe that
\begin{equation}\label{ec3}
\mathcal{P}_{[a:b]}^{\omega_{f/g}}:=g(af_{x}+bf_{y})-f(ag_{x}+bg_{y})=g\mathcal{P}^{df}_{(a,b)} - f\mathcal{P}^{dg}_{(a,b)}.
\end{equation}
\begin{proposition}\label{prop_polar}
Let $C_{[\alpha:\beta]}=\{\alpha f+\beta g=0\}$ be a member of the pencil $\mathcal{P}_{f,g}$. For a generic polar direction $[a:b]$, we have 
\[
I_0\left(\mathcal{P}^{\omega_{f/g}}_{[a:b]},C_{[\alpha:\beta]}\right)=I_0(f,g)+\mu(C_{[\alpha:\beta]})+ \nu (C_{[\alpha:\beta]}) -1,
\]
where $\nu(C_{[\alpha:\beta]})$ denotes the algebraic multiplicity of $C_{[\alpha:\beta]}$ at the origin.
\end{proposition}
\begin{proof}
 First, assume that \(\alpha\ne 0\) and choose a polar direction \([a:b]\) that is generic for both the foliation and the curve \(C_{[\alpha:\beta]}\).
Combining \eqref{ec3}, the properties of the intersection multiplicity and Lemma~\ref{lemma_tei}, we obtain
\[
\begin{array}{lll}
I_0\left(\mathcal{P}^{\omega_{f/g}}_{[a:b]},C_{[\alpha:\beta]}\right) &=&  I_0\left(g\mathcal{P}^{df}_{[a:b]} - f\mathcal{P}^{dg}_{[a:b]},\alpha f + \beta g\right)\\
 &=& I_0\left(\alpha f + \beta g, g\mathcal{P}^{df}_{[a:b]} - f\mathcal{P}^{dg}_{[a:b]}+(\alpha f+\beta g)\frac{\mathcal{P}^{dg}_{[a:b]}}{\alpha}\right)\\
 &=&  I_0\left(\alpha f+ \beta g, \frac{g}{\alpha}\left(\alpha\mathcal{P}^{df}_{[a:b]} +\beta\mathcal{P}^{dg}_{[a:b]}\right)\right)\\
 &=& I_0(f,g)+ I_{0}\left(\alpha f + \beta g, \mathcal{P}^{d(\alpha f+\beta g)}_{[a:b]}\right)\\
 &=& I_0(f,g)+ \mu(C_{[\alpha:\beta]}) + \nu(C_{[\alpha:\beta]})-1.
\end{array}
\]

If $\alpha=0$, then $C_{[\alpha:\beta]}$ is the member $g=0$. Reducing \eqref{ec3} modulo $g$ we obtain  $\mathcal{P}^{\omega_{f/g}}_{[a:b]}\equiv-f\mathcal{P}^{dg}_{[a:b]}$.  Consequently, the argument above yields  $I_0(\mathcal{P}^{\omega_{f/g}}_{[a:b]},g)=I_0(f,g)+\mu(g)+\nu(g)-1$. For any finite collection of specified members of the pencil, a common generic polar direction can be chosen.

\end{proof}

Let $f$ and $g$ be relatively prime germs of reduced holomorphic functions
vanishing at the origin of $\mathbb{C}^2$. The fibers of the pencil $\mathcal{P}_{f,g}$ generated by
$f$ and $g$ are naturally described by the holomorphic foliation defined by the
germ of the $1$-form $\omega_{f/g}=gdf-fdg.$
\par As is well known, this foliation is dicritical and therefore admits a balanced divisor of separatrices at the origin of $\mathbb{C}^{2}$ (see \cite[Section 1]{Genzmer}). Let
\[\mathcal{D}=\sum_{B}a_B\cdot B\]
be such a divisor. We define the {\it polar intersection number} of the pair $(f,g)$ along $\mathcal{D}$ by
\begin{equation}\label{Delta0}
{\Delta}_{0}((f,g),\mathcal{D}):=\sum_{B} a_B \cdot I_{0}(\mathcal{P}^{\omega_{f/g}},B),
\end{equation}
where $\mathcal{P}^{\omega_{f/g}}$ denotes the generic polar curve of the foliation $\F_{\omega_{f/g}}$.

The following corollary shows that \(\Delta_{0}((f,g),\mathcal{D})\) does not depend on the choice  of a balanced divisor of separatrices \(\mathcal{D}\).

\begin{corollary}[{\it Teissier's Lemma for Pairs}]
   For every balanced divisor $\mathcal{D}$ of separatrices for $\F_{\omega_{f/g}}$, we have 
    \[{\Delta}_{0}((f,g),\D)=\mu(f,g)+\nu(\F_{\omega_{f/g}}),\]
    where $\nu(\F_{\omega_{f/g}}):=\min\{\nu(gf_x-fg_x),\nu(gf_y-fg_y)\}$. Moreover, if $\nu_{gen}$ denotes the generic multiplicity of a pencil member, and  $h_1,\ldots,h_r$ are  the distinct bifurcation fibers of $\mathcal{P}_{f,g}$ other than $f$ and $g$, then
\[
\nu(\F_{\omega_{f/g}})=\nu(f)+\nu(g)-1+\sum_{j=1}^r(\nu(h_j)-\nu_{gen}).
\]
\end{corollary}
\begin{proof}

By \cite[Lemma 4.4]{FS-GB-SM1}, the polar intersection number along a balanced divisor of separatrices is equal to the sum of the Milnor number and the algebraic multiplicity of the foliation, minus its tangency-excess terms. Since $\F_{\omega_{f/g}}$ is a generalized curve \cite{Camacho}, all the  tangency-excess terms vanish. This proves the first assertion. Observe that the value  $\Delta_{0}((f,g),\D)$ is independent of the divisor $\D$.
To prove the second formula, we consider the balanced divisor $(fgh_1\cdots h_r)- (\varphi_1\cdots \varphi_r)$ constructed in Section~3 and choose a common generic polar direction for its finitely many fibers. Summing the identities given by Proposition \ref{prop_polar}, weighted by the corresponding  divisor coefficients, and after \eqref{Delta0} and Theorem \ref{th bif},  we obtain
\begin{eqnarray*}
{\Delta}_0((f,g),\mathcal{D})
&=&2I_0(f,g)+\mu(f)+\mu(g)+\sum_{j=1}^r(\mu(h_j)-\mu_{gen})\\
&+&\nu(f)+\nu(g)+\sum_{j=1}^r(\nu(h_j)-\nu_{gen})-2\\
&=&\mu(f,g)+\nu(f)+\nu(g)-1+\sum_{j=1}^r(\nu(h_j)-\nu_{gen}).
\end{eqnarray*}
 Comparing this identity with the first formula yields  the stated expression for   $\nu(\F_{\omega_{f/g}})$.
\end{proof}

Let $J(f,g)=f_xg_y-f_yg_x$ denote the Jacobian curve of the pair $(f,g)$. By Lê's formula (see \cite[Proposition 3.1]{Sz2}), we have
\[I_{0}(f,f_xg_y-f_yg_x)=\mu(f)+I_0(f,g)-1.\]
By symmetry,  $I_0(g,f_xg_y-f_yg_x)=\mu(g)+I_{0}(f,g)-1$.
Adding these two equalities and using the properties of the intersection
multiplicity, we obtain
\begin{equation}\label{eq_teissier}
    I_{0}(fg,f_xg_y-f_yg_x)=\mu(fg)-1.
\end{equation}

\begin{remark}
    Note that if $(f:g)$ is a semitame meromorphic function germ, then it follows from Corollary \ref{caracsemitame1} that $\mu(fg)=\mu(f,g)-\tau(f,g)+\tau(fg)$, which together with \eqref{eq_teissier} implies the relation
    \[
    I_{0}(fg,\, f_xg_y-f_yg_x)=\mu(f,g)-\tau(f,g)+\tau(fg)-1.
    \]
\end{remark}

\par Finally, combining \eqref{eq_teissier} with the bifurcation formula
(Theorem \ref{th bif}), we obtain the following relation
between the Jacobian curve $J(f,g)$ and $\mu(f,g)$:

\begin{proposition}
Let $f,g \in \mathcal O_2$ be relatively prime germs, and assume that both vanish at the origin. Then
\[I_{0}(fg,f_xg_y-f_yg_x)=\mu(f,g)-\sum_{[\alpha:\beta]\in \mathcal{B}_\mu(f,g)^{*}}(\mu(\alpha f+ \beta g)-\mu_{gen})-1,\]
 where $\mathcal{B}_\mu(f,g)^{*}=\mathcal{B}_\mu(f,g)\backslash \{[1:0],[0:1]\}$.  
 \end{proposition}


\section{An upper bound for $\mu(f,g)$}
\label{sec: bound}
Let $f$ and $g$ be relatively prime germs of reduced holomorphic functions, and
let $\mathcal P_{f,g}$ be the pencil generated by them. In this section, we derive an upper bound for $\mu(f,g)$ in terms of the Tjurina number of the product of $f,g$  and the bifurcation fibers of $\mathcal P_{f,g}$, together with the intersection multiplicity $I_0(f,g)$.

\begin{theorem}\label{teocota}
Let $f,g\in\mathcal{O}_2$ be coprime germs vanishing at the origin, and let $h_1,\ldots,h_r$ denote the distinct bifurcation fibers of $\mathcal P_{f,g}$ different from $f$ and $g$. Assume that every member of the pencil is reduced. Then
\[
\mu(f,g)\le
2\tau(fgh_1\cdots h_r)-2r(r+2)I_0(f,g).
\]
\end{theorem}
\begin{proof}
By \cite[Lemma 4.1]{FFMR}, the 1-form $\omega_{f/g}=gdf-fdg$ has an isolated singularity at the origin. The construction in Section~3 gives a primitive balanced divisor $\mathcal{D}=\mathcal{D}_0-\mathcal{D}_{\infty}$ with reduced zero equation $\mathcal{D}_0=fgh_1\cdots h_r$. Since $\F_{\omega_{f/g}}$ admits the meromorphic first integral $f/g$, it is a generalized curve foliation and is therefore of second type. The Brian\c{c}on--Skoda bound \cite[Theorem B]{Skoda} gives
\begin{equation}\label{ecc33}
\mu(f,g)=\mu(\F_{\omega_{f/g}})\leq2\tau(\F_{\omega_{f/g}},\mathcal{D}_0).
\end{equation}

\noindent Applying Lemma~\ref{lemma_tjurina} to the $r+2$ distinct pencil members appearing in the zero divisor $\mathcal D_0$, we obtain

\[
\tau(\F_{\omega_{f/g}},\mathcal{D}_0)=\tau(fgh_1\cdots h_r)-r(r+2)I_0(f,g).
\]
Substituting this expression into \eqref{ecc33} completes the proof.
\end{proof}

\par The following example is a modification of an example of Genzmer given in \cite[p. 2]{Genzmerejemplo}.
\begin{example}
    Let $f=xy+y^2+x^3$, $g=xy$ and let $h=f-g=y^2+x^3$ be the unique bifurcation fiber. 
    Using Singular \cite{Singular}, we compute
    \[
    \mu(f,g)=12,\qquad \tau(fgh)=27,\qquad r=1,\qquad I_0(f,g)=5.
    \]
    Hence, Theorem~\ref{teocota} yields
    \[
    12=\mu(f,g)\leq 2\tau(fgh_1\cdots h_r)-2r(r+2)I_0(f,g)=24.
    \]
\end{example}


\section{On a Dimca Formula for Generic pencils}
\label{sec: Dimca}
Let $C_1: f = 0$ and $C_2 : g = 0$ be reduced algebraic curves of degree $k$ in $\mathbb{P}^2$, and let $\mathcal{P} : \alpha C_1 + \beta C_2$ denote the pencil in $\mathbb{P}^2$ generated by $C_1$ and $C_2$.

Following Dimca \cite[Section 5]{Dimca2017}, we say that $\mathcal{P}$ is \textit{generic} if the curves $C_1$ and  $C_2$ intersect transversely in exactly $ k^2$ points.

Under this assumption, the generic member of $\mathcal{P}$ is smooth, and every member of the pencil $\mathcal{P}$ is smooth at each of the $k^2$ base points. Let $C_1^s,\ldots,C_\ell^s$ denote  the singular members of this pencil  $\mathcal{P}$. Dimca \cite[Proposition 5.1]{Dimca2017} proved the following result using tame regular functions (see \cite[Corollary 3.5]{Siersma95}) and the Euler characteristic of complex constructible sets. We provide an alternative proof based on foliations.
For any reduced algebraic curve $C$ in $\mathbb{P}^2$, we denote $\mu(C):=\sum_p\mu_p(C)$ its \textit{global Milnor number}.
\begin{proposition}\label{prop_pencil}
If the pencil $\mathcal{P}$ is generic, then the sum of the global Milnor numbers of its singular members $C^s_j$  satisfies
\[
\sum_{j=1}^\ell \mu(C_j^s) = 3(k-1)^2.
\]
\end{proposition}

\begin{proof}
Let $\F_{\omega_{f/g}}$ be the foliation defined by $\omega_{f/g}=gdf-fdg$. At each base point $p$, the transversality assumption implies that $u=f$ and $v=g$ form a  local  coordinate system. The foliation is therefore radial, with local equation $vdu-udv=0$, and hence
\begin{equation}\label{eq_mil}
\mu_p(\F_{\omega_{f/g}})=I_p(f,g)=1.
\end{equation}
Now let $p$ be a point outside the base locus. Suppose first that $g(p)\neq0$, and set $c=f(p)/g(p)$. The unique member of the pencil passing through $p$ has local equation  $f/g-c=0$. Since
$\omega_{f/g}=g^2d(f/g)$, the point $p$ is a singular point of the foliation if and only if it is a singular point of this member, and their local Milnor numbers coincide. The same argument applies in the chart where $f(p)\neq0$. In particular, for any singular member $C_j^s$ and any singular point $p\in C_j^s$ outside the base locus, we have
\begin{equation}\label{eq:asterisco}
\mu_p(\F_{\omega_{f/g}})=\mu_p(C_j^s).
\end{equation}

There is a unique member of the pencil passing through each point outside the base locus, and every member is smooth at the base points. Moreover, all members are reduced: any multiple component would meet the base locus by Bézout's theorem, contradicting the smoothness of the members there. Hence the foliation has only isolated singularities and degree $2k-2$. Its total Milnor number is $(2k-2)^2+(2k-2)+1$ (see \cite[p. 19]{Brunella-book}). Summing \eqref{eq_mil} over the $k^2$ base points and \eqref{eq:asterisco} over all remaining singular points, we obtain

\[
(2k-2)^2+(2k-2)+1=k^2+\sum_{j=1}^{\ell}\mu(C_j^s).
\]
Consequently,
\[
\sum_{j=1}^{\ell}\mu(C_j^s)=4k^2-6k+3-k^2=3(k-1)^2,
\]
as required.
\end{proof}

Analogously, we define the {\it global Tjurina number of} $C$ by $\tau(C) := \sum_{p} \tau_p(C)$.

\begin{corollary}\label{eq:coroglobal}
If the pencil $\mathcal{P}$ is generic and $k\geq 2$, then
\[
\frac{9}{4} (k-1)^2 < \sum_{j=1}^\ell \tau(C_j^s) \leq 3(k-1)^2.
\]
Moreover,
\[
\sum_{j=1}^{\ell} \tau(C_j^s) = 3(k-1)^2
\]
if and only if all singularities of the curves $C_j^s$ are quasi-homogeneous.
\end{corollary}

\begin{proof}
For every singular point $p \in C_j^s$, Almir\'on's Theorem \cite[Theorem 3.2]{Almiron} yields
\[
\frac{3}{4} \mu_p(C_j^s) < \tau_p(C_j^s) \quad \hbox{\rm for all }\; j = 1, \cdots, \ell.
\]
Summing over the singular points of $C_j^s,$ we obtain
\[
\frac{3}{4} \mu(C_j^s) < \tau(C_j^s).
\]
Summing again over $j=1,\ldots, \ell$ and applying Proposition \ref{prop_pencil}, we conclude that:
\[
\frac{9}{4} (k-1)^2 < \sum_{j=1}^{\ell} \tau(C_j^s).
\]

\noindent On the other hand, Saito's Theorem \cite{Saito} states that $\tau_p(C_j^s) \le \mu_p(C_j^s)$, with equality if and only if every singularity of $C_j^s$ is quasi-homogeneous. Summing over the singular points of $C_j^s$, we obtain $\tau(C_j^s) \le \mu(C_j^s)$, and hence

\[
\sum_{j=1}^{\ell} \tau(C_j^s) \le \sum_{j=1}^{\ell} \mu(C_j^s) = 3(k-1)^2.
\]
Thus,
\[
\frac{9}{4} (k-1)^2 < \sum_{j=1}^{\ell} \tau(C_j^s) \leq 3(k-1)^2,
\]
where equality in the upper bound holds if and only if every singularity of the curves $C_j^{s}$ is quasi-homogeneous (Saito's Theorem \cite{Saito}).
\end{proof}
\begin{remark}
    An interesting consequence of Corollary \ref{eq:coroglobal} is that, when $k = 2$,
\[
\sum_{j=1}^{\ell} \tau(C_j^s) = 3 = \sum_{j=1}^{\ell} \mu(C_j^s).
\]
Hence, every singularity of the curves $C_j^s$ is quasi-homogeneous by Saito's Theorem \cite{Saito}.
\end{remark}

\vspace{1cm}

\noindent{\bf{Funding}}\\
The authors gratefully acknowledge the support of Universidad de La Laguna (Tenerife, Spain), where part of this work was carried out. This work was partially supported by  the Spanish grant PID2023-149508NB-I00, funded by 
MICIU/AEI
/10.13039/501100011033 and by
FEDER, UE. It was also supported by  the Vicerrectorado de Investigación
(VRI) at the PUCP through grant DFI-2025-PI1275. The first author also acknowledges the support of CNPq-Brazil through Projeto Universal 408687/2023-1 "Geometria das Equa\c{c}\~oes Diferenciais Alg\'ebricas" and the PQ fellowship CNPq-Brazil 306011/2023-9.\\

\noindent{\bf{Acknowledgments}}\\
 Arturo Fernández Pérez and Nancy Saravia Molina would like to thank the Universidad de La Laguna for its hospitality, where this work was completed. 
 The authors are grateful to the referee for their careful reading of the manuscript and for their insightful comments and valuable suggestions, which helped improve the quality and presentation of this work.

\vspace{1cm} 

\noindent{\bf Data Availability Statement:} 
Data sharing is not applicable to this article as no data sets were generated or analyzed during the current study.
\vspace{1cm}

\noindent{\bf Declarations
Conflict of Interest:} The authors declare that they have no conflict of interest.

\end{document}